\documentclass{amsart}
\usepackage{amsmath,amsthm,amssymb}
\usepackage{marvosym}
\usepackage{graphicx}
\usepackage{color}
\usepackage{epsfig}
\usepackage{morefloats}
\usepackage{hyperref}
\usepackage{fullpage}
\usepackage[active]{srcltx}

\theoremstyle{plain}
\newtheorem*{thm}{Theorem}

\newcommand{\comment}[1]{}

\newcommand{\bdry}{\ensuremath{\partial}}

\title{The Poincar\'e Homology Sphere \\and\\ Almost-Simple Knots in Lens Spaces}

\author{Kenneth L.\ Baker}
\thanks{This work is partially supported by grant \#209184 to Kenneth L.\ Baker from the Simons Foundation.}
\address{
Department of Mathematics
University of Miami,
PO Box 249085
Coral Gables, FL 33124-4250}
\email{k.baker@math.miami.edu}
\urladdr{http://math.miami.edu/\char126 kenken}

\begin{document}

\begin{abstract}
Hedden defined two knots in each lens space that, through analogies with their knot Floer homology and doubly pointed Heegaard diagrams of genus one, may be viewed as generalizations of the two trefoils in $S^3$.  Rasmussen shows that when the `left-handed' one is in the homology class of the dual to a Berge knot of type VII, it admits an L-space homology sphere surgery.  In this note we give a simple proof that these L-space homology spheres are always the Poincar\'e homology sphere.

\end{abstract}
\maketitle

Beginning with the standard genus $1$ Heegaard diagram for the lens space $L(p,q)$ with $p>q>0$ where an $\alpha$ curve and a $\beta$ curve intersect $p$ times, there are just two simple isotopies of these curves that produce a genus one Heegaard diagram with $p+2$ intersection points.  Figure~\ref{fig:heegaardtwist} adds two basepoints to a local view of each of the standard diagram (center) and the two isotoped diagrams (left and right) to form doubly pointed Heegaard diagrams for knots in the lens space.  The center knot is the {\em simple} knot $K(p,q,q+1)$ in the language and notation of \cite{rasmussen}. (Such a knot is called {\em grid number one} in \cite{bgh}.)  The knots to the left and right are Hedden's knots $T_L$ and $T_R$ \cite{hedden}.  The knots $T_L$ and $T_R$ are each homologous to $K(p,q,q+1)$ since they each differ from $K(p,q,q+1)$ by a crossing change as demonstrated in Figure~\ref{fig:3d!}.  Due to their knot Floer homologies being as close to that of the simple knots without actually being so, we refer to these two knots as {\em almost-simple}.

\begin{figure}
\centering
\includegraphics[width=4.75in]{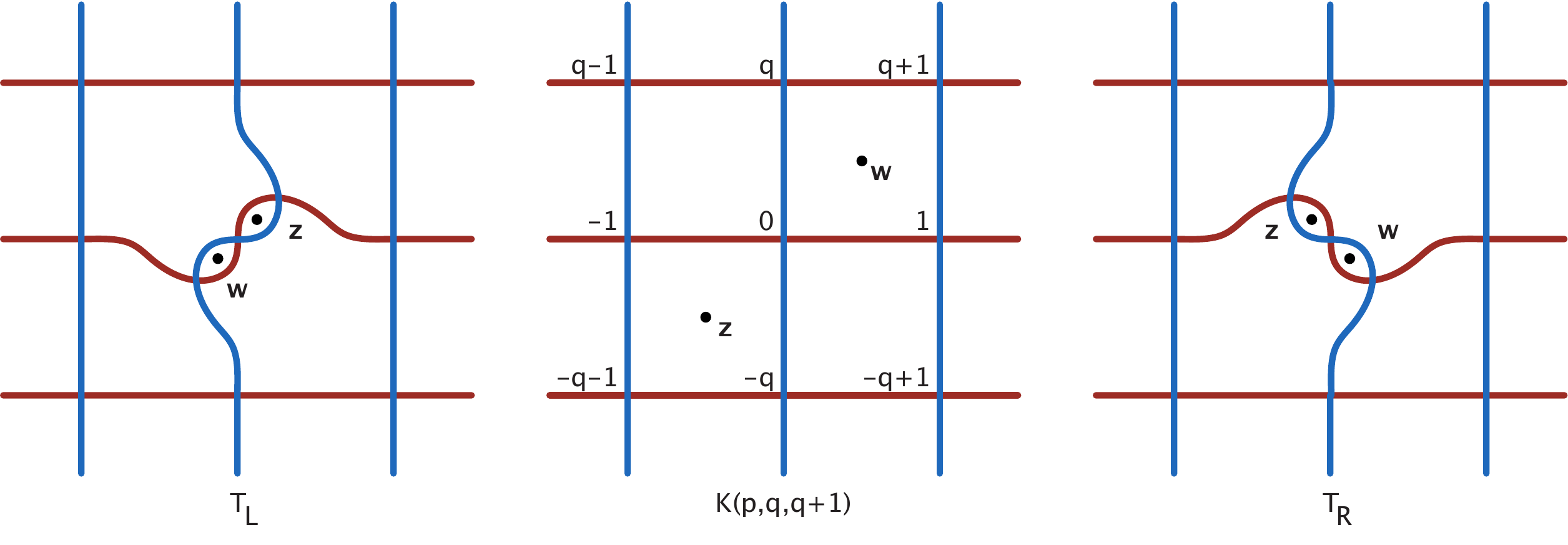}
\caption{}
\label{fig:heegaardtwist}
\end{figure}

Rasmussen shows that exactly when $K(p,q,q+1)$ is the dual to a Berge knot of type VII (i.e.\ when $q^2+q+1 \equiv 0 \mod p$) does $T_L$ admit an integral Dehn surgery producing an integral homology sphere that is also an L-space \cite[\S 5]{rasmussen}.  He states that he determined the surgered manifolds for these $T_L$ in lens spaces as large as $L(39,16)$ to actually be the Poincar\'e homology sphere $P^3$ by checking their fundamental groups.  That all these $T_L$ that are homologous to duals of Berge knots of type VII indeed do admit Poincar\'e homology sphere surgeries is then one part of Greene's Conjecture~1.9 \cite{greene}.  Here we provide a short proof.

\begin{thm}
If $T_L$ is homologous to the dual to a Berge knot of type VII, then it admits an integral surgery producing the Poincar\'e homology sphere.
\end{thm}

\begin{proof}[Proof 1]
Let $J$ be the positive (left-handed) trefoil, and let $F_0$ and $F_\pi$ be two of its fibers.  Then $F_0 \cup F_\pi$ is a genus $2$ Heegaard splitting of $S^3$.  Now consider an essential simple closed curve in $F_0$ as a knot $K$ embedded in $F_0$ in $S^3$.  Then $K$ is immediately a doubly primitive knot (\cite{berge-lens}) with respect to this Heegaard splitting .  The knots in $S^3$ that embed as essential curves on the once-punctured torus fiber of the positive trefoil and their mirrors are the Berge knots of type VII \cite{berge-lens}. 

One may show the surface framing that $F_0$ induces on $K$ has positive integral slope. Dehn surgery along this slope then yields a lens space $L(p,q)$ and the knot $K'=K(p,q,q+1)$ dual to $K$.  The surgery on $K$ then causes a surgery of $F_0$ along $K$, rendering it a disk $D_0$ that $K'$ intersects twice.    Indeed $D_0 \cup F_\pi$ is a Heegaard torus for the lens space $L(p,q)$, and, due to the doubly primitive presentation of $K$, $K'$ is an unknotted arc in the solid torus to on each side.  Notice that in $L(p,q)$, after surgery on $K$, the positive trefoil $J$ is now the boundary of this disk: $J = \bdry D_0$.  Therefore $+1$ surgery on $J$ in $L(p,q)$ simply inserts a full left-handed twist in $K'$.  We will conclude that this twist transforms $K'$ into $T_L$, though for now let us call the resulting knot $T'$.

The Poincar\'e homology sphere $P^3$ may be obtained as $+1$ surgery on $J$ in $S^3$.  Let us first do this surgery, regard $K \subset F_0$ as a knot in $P^3$, and then do $F_0$--framed surgery on $K$.  Because we are only swapping the order in which surgery is done on the link $K \cup J \subset S^3$, this must also produce the lens space $L(p,q)$ in which the knot dual to $K \subset P^3$ is the above knot $T'$, i.e.\ $K'$ with that full left-handed twist.  In particular, we conclude that $T' \subset L(p,q)$ is a $1$-bridge knot with respect to the Heegaard torus  $D_0 \cup F_\pi$, is in the homology class of a knot dual to a Berge knot of type VII, and admits an integral Poincar\'e homology sphere surgery.  

Because the Poincar\'e sphere $P^3$ is both an L-space and an integral homology sphere, \cite[Proposition 4.5]{rasmussen} implies that since $T'$ is a knot in the lens space $L(p,q)$ representing a primitive homology class (that of $K'=K(p,q,q+1)$) and admitting an integral surgery to $P^3$, then $\rm{rk}(\widehat{HFK}(L(p,q),T))\leq p+2$.  Because $T'$ is $1$-bridge but generically not the simple knot $K(p,q,q+1)$ (otherwise $T'=K(p,q,q+1)$ is actually dual to the trefoil $J$, the only knot in $S^3$ with an integral surgery to $P^3$ \cite{ghiggini}), we must have that $T'=T_L$ by \cite[Proposition 3.3]{hedden}.
\end{proof}

If one were to pay more attention to the meridional disks for handlebodies to either side of $F_0 \cup F_\pi$ that are disjoint from $K$, one could directly see that the twist on $J$ in $L(p,q)$ takes $K'=K(p,q,q+1)$ to $T_L$ without appealing to the knot Floer homology argument in the last paragraph of the proof above.  Indeed, $J$ is the loop in the local picture of Heegaard torus in the center of Figure~\ref{fig:heegaardtwist} that encircles the two basepoints.
Figure~\ref{fig:3d!} then illustrates that $T_L$ is obtained from $K(p,q,q+1)$ by a full Dehn twist on this curve.

\begin{figure}
\centering
$
\begin{array}{ccc}
\includegraphics[width=2.5in]{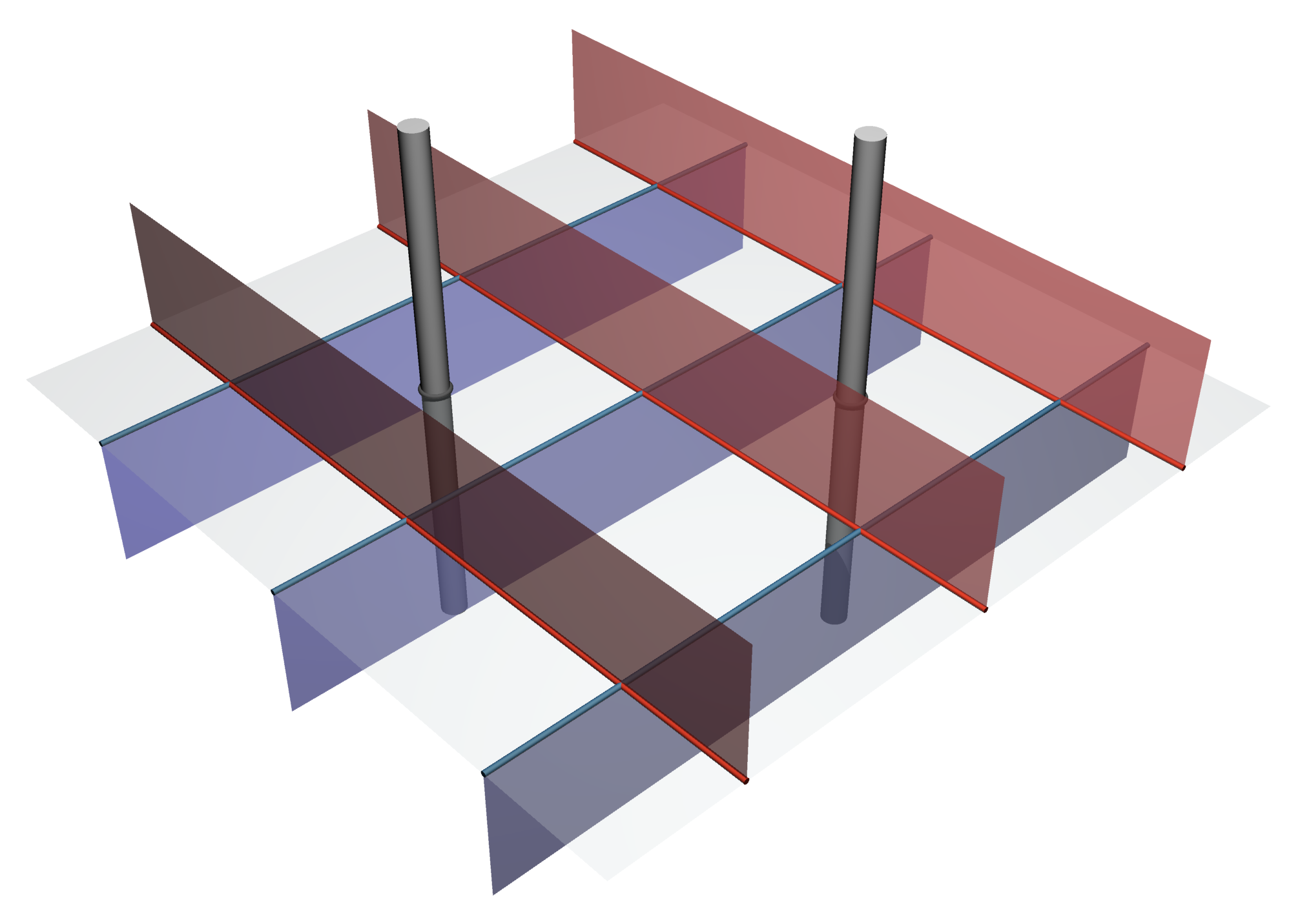} & & \includegraphics[width=2.5in]{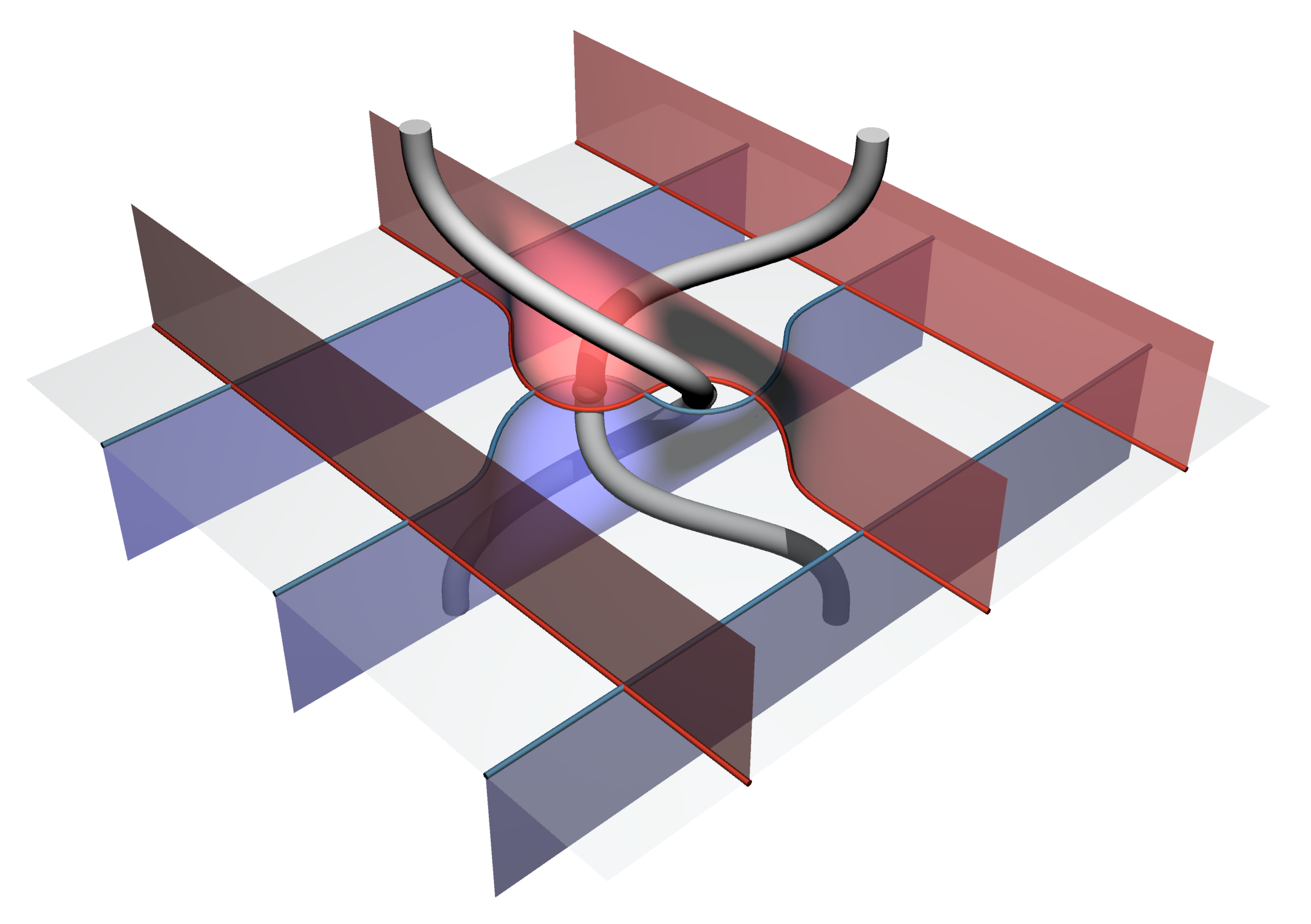}\\
K(p,q,q+1) && T_L
\end{array}
$
\caption{}
\label{fig:3d!}
\end{figure}

\smallskip
The manifolds, Heegaard splittings, knots, and fibers in the proof above all behave well with respect to certain involutions.
Let us therefore recast that proof in terms of bandings of between a two-bridge link $\ell$ and each an unknot $\hat{s}$ and the $(-3,5)$--torus knot $\hat{p}$.  We omit details of the correspondences.
\begin{proof}[Proof 2]
Let $A$ be the braid axis of the closure $\hat{s}$ of the $3$--braid $\sigma_1 \sigma_2$.
The positive trefoil $J$ may be viewed as the lift to $S^3$ of $A$ in the double cover of $S^3$ branched over the unknot $\hat{s}$.  A disk $D$ bounded by $A$ that intersects $\hat{s}$ three times lifts to a once-punctured torus fiber $F$ of $J$.  Any essential simple closed curve on $F$ is then isotopic on $F$ to a curve that is the lift of some simple arc $k$ on $D$ connecting two of the three points of $D \cap \hat{s}$.   We may simplify the presentation of this arc with a $3$--braid conjugation.  The top-left of Figure~\ref{fig:tangle} shows this setup where $\gamma$ is the conjugating $3$--braid.  Banding together the unknot $\hat{s}$ along $k$ with the framing bestowed upon it by $D$, we obtain a two-bridge link $\ell$ and a dual arc $k'$ as shown in the top-middle and top-right of Figure~\ref{fig:tangle}.  Observe that $A$ bounds a disk that intersects each $\ell$ and $k'$ just once.   

Performing $+1/2$--Dehn surgery along $A$ before the banding takes the closed $3$--braid $\hat{s}$ to $\hat{p}$, the closure of the $3$--braid $\sigma_1 \sigma_2 (\sigma_1 \sigma_2)^{-6} = (\sigma_1 \sigma_2)^{-5}$, as in the lower-left of Figure~\ref{fig:tangle}.  The closed braid $\hat{p}$ is the $(-3,5)$--torus knot and its double branched cover is $P^3$. (The pretzel knot $P(2,-3,-5)$ is isotopic to the $(-3,5)$--torus knot, see Figure~\ref{fig:isotopy}.)   Doing this surgery on $A$ after the banding does not change $\ell$, but does cause $a'$ to wrap around $\ell$ twice near $\ell \cap D$ giving an arc $t'$ as in the lower-right of Figure~\ref{fig:tangle}.  Minding framings, a banding of the two-bridge link $\ell$ along this arc $t'$ produces the $(-3,5)$--torus knot $\widehat{p}$.
\end{proof}

\begin{figure}
\centering
\includegraphics[width=4.5in]{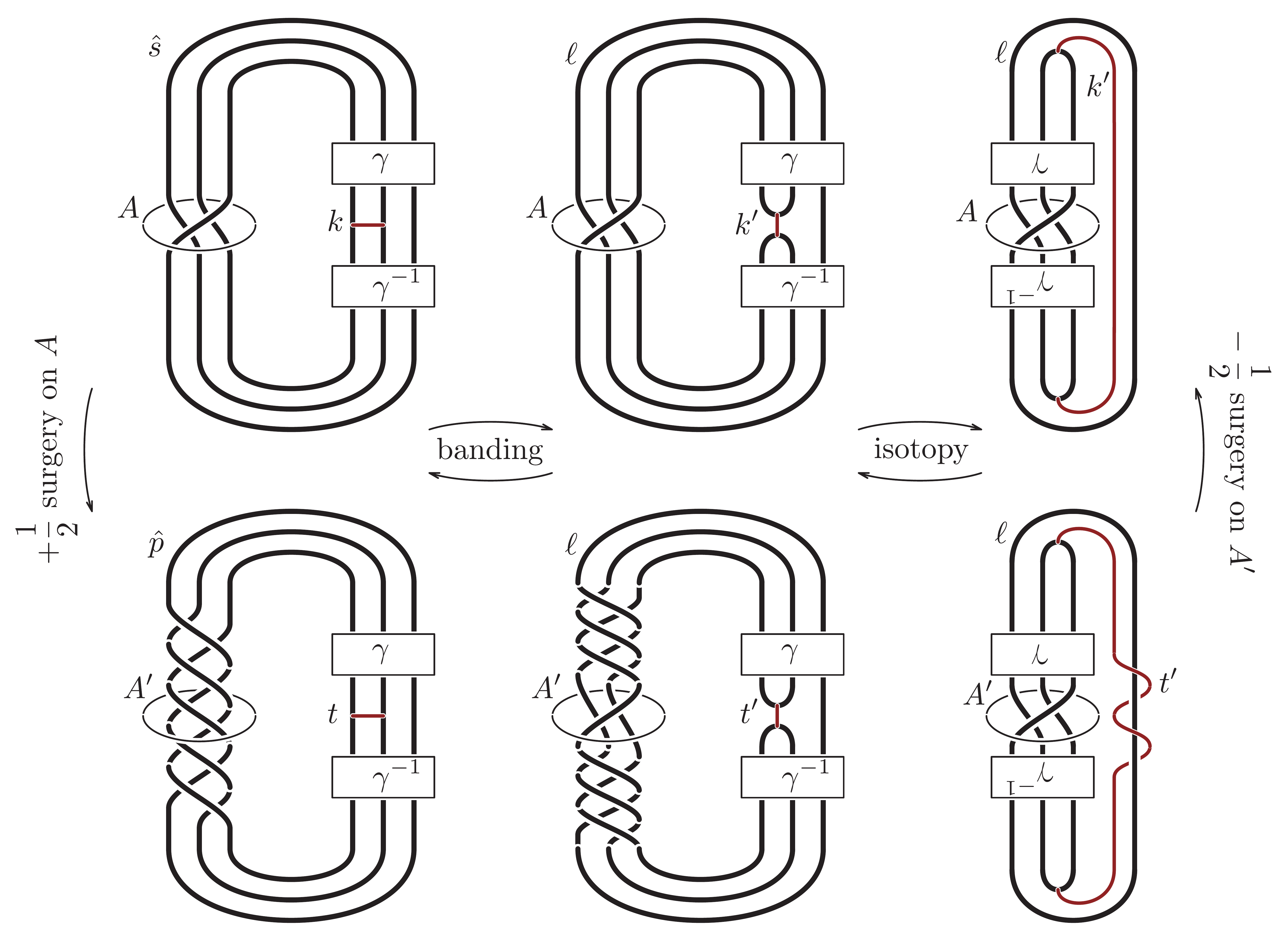}
\caption{}
\label{fig:tangle}
\end{figure}

\begin{figure}
\centering
\includegraphics[width=5.5in]{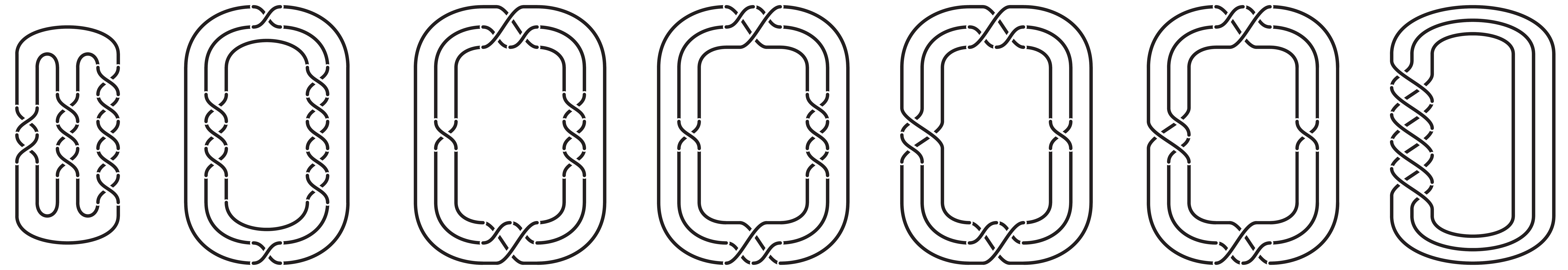}
\caption{}
\label{fig:isotopy}
\end{figure}

\bibliographystyle{amsplain}
\bibliography{trefoil.bib}

\end{document}